\documentclass[11pt]{amsart}
\usepackage{amsmath, amsrefs, amsthm, amssymb}
\usepackage{graphicx}
\usepackage[active]{srcltx}		
\usepackage{pdfsync}					
\newcommand{\dd}{\,{\rm d}}

\newcommand\R{{\mathbb{R}}}

\renewcommand\S{{\mathbb{S}}}

\newtheorem{theorem}{Theorem}[section]
\newtheorem{proposition}[theorem]{Proposition}

\theoremstyle{definition}

\theoremstyle{remark}

\numberwithin{equation}{section}
\numberwithin{figure}{section}

\begin{document}

\title[Liouville Theorem for Degasperis--Procesi]
{A Liouville theorem for the Degasperis-Procesi equation}

\author{Lorenzo Brandolese}

\address{L. Brandolese: Universit\'e de Lyon~; Universit\'e Lyon 1~;
CNRS UMR 5208 Institut Camille Jordan,
43 bd. du 11 novembre,
Villeurbanne Cedex F-69622, France.}
\email{Brandolese{@}math.univ-lyon1.fr}
\urladdr{http://math.univ-lyon1.fr/$\sim$brandolese}

\thanks{Supported by the ANR project DYFICOLTI ANR-13-BS01-0003-01}




\date{October, 23 2014}

\begin{abstract}
We prove that the only global, strong, spatially periodic solution to the Degasperis-Procesi equation,
 vanishing at some point $(t_0,x_0)$,
 is the identically zero solution. We also  establish the analogue of such Liouville-type theorem
for the Degasperis-Procesi equation with an additional dispersive term.
\end{abstract}

\maketitle

\section{Introduction and main results}
\label{sec:intro}
We study spatially periodic solutions of the Degasperis-Procesi
equation
\begin{equation}
 \label{DP0}
 \begin{cases}
u_t-u_{txx}+4uu_x=3u_xu_{xx}+uu_{xxx}, &t>0,\quad x\in\R\\
  u(t,x)=u(t,x+1), &t\ge0,\quad x\in\R.
 \end{cases}
\end{equation}
 
Such equation attracted a considerable interest in the past few years, both for its remarkable mathematical 
 properties (see, {\it e.g.\/} \cites{DP99, CIL2010}), and for its physical interpretation as  an asymptotic
model obtained from the water-wave system in shallow water regime. In this setting, the equation models 
moderate amplitude waves and  $u$ stands for a horizontal velocity of the water at a fixed depth,
see \cites{AConL09, John02} and the references therein for further physical motivations.

The Cauchy problem associated with~\eqref{DP0} can be more conveniently reformulated as
\begin{equation}
 \label{DP}
 \begin{cases}
  u_t+uu_x+\partial_x p*\biggl(\displaystyle\frac32 u^2\biggr)=0, &t>0,\quad x\in\S\\
  u(0,x)=u_0(x), &x\in\S\\
 \end{cases}
\end{equation}
where $\S$ is the circle and $p$ the kernel of  $(1-\partial_x^2)^{-1}$, given by
 the continuous $1$-periodic function
\begin{equation}
 \label{kernel}
 p(x)=\frac{\cosh(x-[x]-1/2)}{2\sinh(1/2)}.
\end{equation}
Another possible reformulation of~\eqref{DP0} is the  momentum-velocity equation,
\begin{equation}
\label{eq:pot}
y_t+uy_x+3u_xy=0, \qquad y(t,x)=y(t,x+1),\qquad t>0,\;x\in\R,
\end{equation}
where $y=u-u_{xx}$ is the associated potential, and $u$ can be recovered from~$y$ from the convolution relation
$u=p*y$.

It is well known (see, {\it e.g.\/}, \cite{Yin03}) that if $u_0\in H^s(\S)$, with $s>3/2$,
then the problem~\eqref{DP} possess a unique solution
\begin{equation}
 \label{funspa}
 u\in C([0,T),H^s(\S))\cap C^1([0,T),H^{s-1}(\S)),
\end{equation}
for some $T>0$, depending only $u_0$.

The maximal time $T^*$ of the above solution can be finite or infinite.
For instance, if the initial potential $y_0=y(0,\cdot)$ does not change sign, then it is known that $T^*=+\infty$,
see \cite{ELY07}. 
On the other hand, several different  blowup criteria were established, {\it e.g.\/}, in~\cites{Yin03, ELY07,GGL-2011}:
in the shallow water interpretation the finite time blowup corresponds to a wave breaking mechanism, as near the blowup time
solutions remain bounded, but have an unbounded slope in at least one point.

The purpose of this short paper it to establish the following Liouville-type theorem:
\begin{theorem}
\label{th:liou}
The only global solution $u\in  C([0,+\infty),H^s(\S))\cap C^1([0,+\infty),H^{s-1}(\S))$, with $s>3/2$,
to the Degasperis-Procesi equation vanishing at some point $(t_0,x_0)$ is the identically zero solution.
\end{theorem}

The Degasperis-Procesi equation is often written with the additional dispersive term $3\kappa u_x$ in the left-hand side of equation~\eqref{DP0}, where $\kappa\in\R$ is the dispersion parameter.
In this more general setting the above theorem can be reformulated as follows:

\begin{theorem}
\label{th:k}
Let $s>3/2$. If $v\in  C([0,+\infty),H^s(\S))\cap C^1([0,+\infty),H^{s-1}(\S))$ is a global solution
to the Degasperis-Procesi equation with dispersion
\begin{equation}
\label{eq:v}
v_t+vv_x+\partial_x p*\biggl(\displaystyle\frac32 v^2+3\kappa v\biggr)=0, \qquad t>0,\quad x\in\S,
\end{equation}
such that $v(t_0,x_0)=-\kappa$ at some point $(t_0,x_0)$, then $v(t,x)\equiv -\kappa$  for all $(t,x)$.
\end{theorem}

In the next section we will compare these theorems with earlier related results. 
The main idea of the present paper will be remark that, in the dispersionless case, 
for all time $t\in\R^+$, at least one of the two functions 
$x\mapsto e^{\pm \sqrt{3/2}\,q(t,x)} u(t,q(t,x))$, where $q(t,x)$ is the flow of the  global solution~$u$,
must be monotonically increasing.

\medskip
As a byproduct of our approach, 
we will get a  simple and natural blowup criterion for periodic solutions to the Degasperis-Procesi equation,
with or without dispersion, that is of independent interest:

\begin{proposition}
\label{prop:blow}
Let $v_0\in H^s(\S)$, with $s>3/2$, be such that $v_0'(a)<-\textstyle\sqrt{\frac32} \,\bigl|v_0(a)+\kappa\bigr|$
for some $a\in\S$. Then the solution $v\in C([0,T^*),H^s(\S))\cap C^1([0,T^*),H^s(\S))$  of~\eqref{eq:v} arising from~$v_0$
blows up in finite time.
\end{proposition}

%

\section{Comparison with some earlier results}
\label{sec:comparison}

In \cite{ELY07}*{Theorem 3.8}, Escher, Liu and Yin established the blowup for equation~\eqref{DP0}
assuming that $u_0\in H^s(\S)$, $u_0\not\equiv0$,
 and that the corresponding solution $u(t,x)$ \emph{vanishes in at least one point $x_t\in\S$ for all $t\in[0,T^*)$}.
Theorem~\ref{th:liou} improves their result (and the corresponding corollaries) 
by providing the same conclusion $T^*<\infty$ with a shorter proof, and 
under an assumption that is easier to check.

Applying Proposition~\ref{prop:blow} with $\kappa=0$ and $a=0$ 
improves Yin's blowup criterion~\cite{Yin03}*{Theorem 3.2}, establishing the blowup for \emph{odd} initial data
with negative derivative at the origin.

In the particular case $\kappa=0$, Proposition~\ref{prop:blow} improves and  simplifies
the wave-breaking criterion of \cite{ELY07}*{Theorem~4.3} (and its corollaries), that established the blowup under a condition of the form 
$v_0'(a)<-\bigl(c_0\|v_0\|_{L^\infty}+c_1\|v_0\|_{L^2}\bigr)$, with suitable $c_0,c_1>0$. 
In fact, Proposition~\ref{prop:blow} shows that one can take $c_1=0$, and more importantly, one only needs to check the behavior of $u_0$ in a neighborhood of a single point to get the blowup condition.

For general $\kappa\in\R$, Proposition~\ref{prop:blow} extends and considerably simplifies the  blowup condition 
$v_0'(a)<-M$ established in~\cite{GGL-2011}*{Theorem~4.1}, 
where $M=M(\kappa,\|v_0\|_{L^2},\|v_0\|_{L^\infty})$ was given by a quite involved expression.
In the same way, the pointwise estimates~\eqref{pointwise} below allow to improve results like 
\cite{GGL-2011}*{Theorem~4.2} and its corollaries.

A Liouville-type theorem in the same spirit as Theorem~\ref{th:liou} has been established for periodic solutions
of the hyperelastic rod equation in~\cite{BraCor14}, when the physical parameter~$\gamma$ of the model belongs to a 
suitable range (including $\gamma=1$, that corresponds to the dispersionless Camassa--Holm equation).
The specific structure of the nonlocal term
of the Degasperis-Procesi equation makes possible the much more concise proof presented here.
Neverthless, the Degasperis--Procesi equation remains worse understood than Camassa--Holm's.
First of all, the geometric picture between the two equations is different:
the Camassa--Holm equation can be realized as a metric Euler equation. On the other hand,  there is no Riemannian metric on $\text{Diff}^\infty(\S)$ such that the corresponding geodesic flow is given by the Degasperis--Procesi equation.
In fact, the more subtle geometric picture for the latter equation has been disclosed only recently, see \cite{KolEsch11}.
Moreover, no \emph{necessary and sufficient\/} condition for the global existence of solutions to the Degasperis--Procesi equation is available. (Such a condition is instead known for the  Camassa--Holm equation, see \cite{McKean04}). 
For this reason,   Proposition~\ref{prop:blow}  provides valuable information.

\section{Proofs}

\begin{proof}[Proof of Theorem~\ref{th:liou}]
Equation~\eqref{DP0} is invariant under time translations and under the transformation
$\tilde u(t,x)=-u(-t,x)$. Therefore, it is enough to prove that if $u_0(x_0)=0$ at some point $x_0\in\S$, but 
$u_0\not\equiv0$, then the solution~$u\in C([0,T),H^s(\S))\cap C^1([0,T),H^{s-1}(\S))$ 
arising from~$u_0$ must blow up in finite time.
 Let $\alpha\in(x_0,x_0+1)$ be such that $u_0(\alpha)\not=0$.
 
 We first consider the case $u_0(\alpha)>0$.
 Let us introduce the map 
 \[\phi(x)=e^{\sqrt{\textstyle\frac32} x}u_0(x).\]
 By the periodicity and the continuity of~$u_0$, we can find an open interval $(\alpha,\beta)\subset (x_0,x_0+1)$
 such that $\phi(x)>0$ on the interval $(\alpha,\beta)$ and $\phi(\alpha)>0$, $\phi(\beta)=0$.
 An integration by parts gives
 \begin{equation}
  \label{iip}
  \int_\alpha^{\beta} e^{\sqrt{\textstyle\frac32}x}u_0'(x)\dd x
  =\phi(\beta)-\phi(\alpha)-\int_\alpha^{\beta}\sqrt{\textstyle\frac32}\phi(x)\dd x.
 \end{equation}
We deduce from this the existence of~$a\in(\alpha,\beta)$ such that 
$u_0'(a)<-\sqrt{\textstyle\frac32}\,u_0(a)<0$.
Indeed, otherwise, we could bound the left-hand side in~\eqref{iip}  from below by 
$\displaystyle\int_\alpha^\beta -\sqrt{\textstyle\frac32}\phi(x)\dd x$,
and get the contradiction $\phi(\alpha)\le \phi(\beta)$.

The second case to consider is $u_0(\alpha)<0$:
introducing now the map $\psi(x)=e^{-\sqrt{\textstyle\frac32} x}u_0(x)$ and arguing as before, 
we get in this case the existence of a point $a$ such that
$u_0'(a)<\sqrt{\textstyle\frac32} \,u_0(a)<0$.
Notice that in both cases we get
\begin{equation}
 \label{blow-cond}
\exists \,a\in\S \quad\mbox{such that}\quad
u_0'(a)<-\sqrt{\textstyle\frac32}\,\bigl|u_0(a)\bigr|.
\end{equation}
We thus reduced the proof of our claim to establishing the finite time blowup under condition~\eqref{blow-cond}, with
$u_0\in H^s(\R)$.  In fact, by approximating $u_0$ with a sequence  $(u_n)\subset H^3(\S)$, we can assume without loss of generality that $u_0\in H^3(\S)$. (Indeed, the argument below will provide un upper bound for $T^*$ independent on 
the parameter~$n$).

\medskip
Let us introduce the flow map
\begin{equation}
 \label{def:flow}
 \begin{cases}
  q_t(t,x)=u(t,q(t,x)), &t\in (0,T^*),\quad x\in\R\\
  q(0,x)=x, &\qquad x\in\R.
 \end{cases}
\end{equation}
We also introduce the $C^1$ functions, defined on $(0,T^*)$,
\[
 f(t)=(-u_x+\sqrt{\textstyle\frac32} u)\bigl(t,q(t,a)\bigr),
 \]
 and
 \[
 g(t)=-(u_x+\sqrt{\textstyle\frac32} u)\bigl(t,q(t,a)\bigr).
 \]
Taking the spatial derivative in equation~\eqref{DP},
recalling that $(1-\partial^2_x)p$ is the Dirac mass,
we get
\[
 u_{tx}+uu_{xx}=-u_x^2+\frac32 u^2 -p*\biggl(\frac32 u^2\biggr).
\]
Using the definition of the flow map~\eqref{def:flow} we obtain
\begin{equation*}
 \begin{split}
  f'(t)
  &=\biggl[ -(u_{tx}+uu_{xx})+\sqrt{\textstyle\frac32}(u_t+uu_x)\biggr](t,q(t,a))\\
  &=\biggl[u_x^2-\textstyle\frac32 u^2 +(p-\sqrt{\textstyle\frac32}p_x)*(\textstyle\frac32 u^2)\biggr](t,q(t,a)).\\
   \end{split}
\end{equation*}
From expression~\eqref{kernel} we easily get
\[
 (p\pm\beta p_x)\ge0 
 \qquad
 \mbox{if and only if}
 \qquad |\beta|\le {\coth{(1/2)}},
\]
and so, in particular
\begin{equation} 
 \label{fail}
p\pm\sqrt{\textstyle\frac32}p_x\ge0.
 \end{equation}
Hence we get
\begin{equation*}
 \begin{split}
  f'(t)\ge \bigl[ u_x^2 -\textstyle\frac32 u^2 \bigr](t,q(t,a)).
   \end{split}
\end{equation*}
Factorizing the right-hand side leads to the differential inequality
\begin{equation}
\label{di1}
f'(t)\ge f(t)g(t), \qquad t\in(0,T^*).
\end{equation}
A similar computation yields
\begin{equation}
\label{di2}
 g'(t)\ge f(t)g(t), \qquad t\in(0,T^*).
\end{equation}

Let
\[h(t)=\sqrt{f(t)g(t)}.\]
We first observe that
\[
h(0)=\sqrt{fg(0)}=\sqrt{u_0'(a)^2-\textstyle\frac32 u_0(a)^2}\,>0.
\]
Moreover, we deduce from the system~\eqref{di1}-\eqref{di2}, 
applying the geometric-arithmetic mean inequality, that
\[
 h'(t)\ge h^2(t), \qquad t\in(0,T^*).
\]
This immediately implies 
$T^*\le 1/h(0)<\infty$.
Theorem~\ref{th:liou} is completely established.
\end{proof}

\medskip
The above proof establishes also Proposition~\ref{prop:blow} in the particular case $\kappa=0$.
But $u(t,x)=v(t,x-\kappa t)+\kappa$,
is a global solution of~\eqref{DP} if and only if $v$ is a global solution of~\eqref{eq:v} with $u_0=v_0+\kappa$.
Hence, Theorem~\ref{th:k} follows immediately from Theorem~\ref{th:liou}. In the same way, we see that the claim of 
Proposition~\ref{prop:blow} holds true for general~$\kappa$.

\medskip
Our arguments also reveal that global solutions must satisfy quite stringent pointwise estimates.
Indeed, assume that $u\in C([0,\infty),H^s(\S))\cap C^1([0,\infty),H^{s-1}(\S))$ is a given global solution of~\eqref{DP0}.
Then, by our theorem, $\mbox{sign(u)}=1,0$ or $-1$ is well defined and independent on $(t,x)$.  
Moreover,  $u'(t,x)\ge -\sqrt{\frac32} |u(t,x)|$ for all $t\ge0$ and $x\in\S$.
Then, arguing as in~\eqref{iip}, we deduce that, for all $t\ge0$, the map
$x\mapsto e^{\mbox{\scriptsize sign}(u)\sqrt{\frac32} x} \,u(t,x)$ is increasing.
Combining this with the periodicity, we get the pointwise estimates for $u(t,x)$,
for all $t\ge0$, all $\alpha\in\R$ and $\alpha\le x\le \alpha+1$:
\begin{equation}
\label{pointwise}
  e^{\mbox{\scriptsize sign}(u)\sqrt{\frac32}(\alpha-x)}{u(t,\alpha)} \le u(t,x) 
  \le e^{\mbox{\scriptsize sign}(u)\sqrt{\frac32}(\alpha+1-x)}{u(t,\alpha)}.
\end{equation}
From~\eqref{pointwise} one immediately deduces the corresponding estimates for global
 solutions to the Degasperis-Procesi
equation with dispersion.

\medskip
We conclude observing that our results 
seem to  be specific to periodic solutions and have no analogue in the case of solutions in~$ H^s(\R)$.
A reason for this is that in the non-periodic case the expression of $p$ should be 
modified into $p(x)=\frac12 e^{-|x|}$ and the fundamental inequality~\eqref{fail} is no longer true.


\begin{thebibliography}{11}



\bib{BraCMP}{article}{
   author={Brandolese, L.},
   title={Local-in-space criteria for blowup in shallow water and dispersive rod equations},
   journal={Comm. Math. Phys.},
   volume={330},
   year={2014},
   number={1},
   pages={401--414}
  }


  \bib{BraCor14}{article}{
  author={Brandolese, L.},
  author={Cortez, M.~F.},
  title={On permanent and breaking waves in hyperelastic rods and rings},
  journal={J. Funct, Anal.},
  volume={266},
  date={2014},
  number={12},
  pages={6954--6987},
  }  


\bib{Chen2011}{article}{
   author={Chen, W.},
   title={On solutions to the Degasperis-Procesi equation},
   journal={J. Math. Anal. Appl.},
   volume={379},
   date={2011},
   number={1},
   pages={351--359},
}

\bib{Coc-K2006}{article}{
   author={Coclite, Giuseppe M.},
   author={Karlsen, Kenneth H.},
   title={On the well-posedness of the Degasperis-Procesi equation},
   journal={J. Funct. Anal.},
   volume={233},
   date={2006},
   number={1},
   pages={60--91},
}


\bib{CIL2010}{article}{
   author={Constantin, A.},
   author={Ivanov, R.I.},
   author={Lenells, J.},
   title={Inverse scattering transform for the Degasperis-Procesi equation},
   journal={Nonlinearity},
   volume={23},
   date={2010},
   number={10},
   pages={2559--2575},
}

\bib{AConL09}{article}{
   author={Constantin, A.},
   author={Lannes, D.},
   title={The hydrodynamical relevance of the Camassa-Holm and
   Degasperis-Procesi equations},
   journal={Arch. Ration. Mech. Anal.},
   volume={192},
   date={2009},
   number={1},
   pages={165--186},
  }

\bib{DP99}{article}{
 author={De Gasperis, A.},
 author={Procesi, M.},
 title={Asymptotic integrability. In: Symmetry and perturbation theory. A. De Gasperis amd G. Gaeta editors.},
 journal={World Scientific},
 year={1999},
 pages={1463--1474},
}


   \bib{KolEsch11}{article}{
   author={Escher, J.},
   author={Kolev, B.},
   title={The Degasperis-Procesi equation as a non-metric Euler equation},
   journal={Math. Z.},
   volume={269},
   date={2011},
   number={3-4},
   pages={1137--1153},
}
  
\bib{ELY07}{article}{
 author={Escher, J.},
 author={Liu, Y.},
 author={Yin, Z.},
 title={Shock Waves and Blow-up Phenomena for the Periodic Degasperis-Procesi Equation},
 journal={Indiana Univ. Math. J.},
 volume={56},
 number={1},
 year={2007},
 pages={87--117},
}  
  
  \bib{GGL-2011}{article}{
  author={Gui, F.},
  author={Gao, H.},
  author={Liu, Y},
  title={Existence of permanent and breaking waves for the periodic Degasperis--Procesi equation with linear dispersion},
  journal={J. Reine angew. Math.},
  volume={657},
  year={2011},
  pages={199--223},
  }
  


  \bib{John02}{article}{
  author={Johnson, R. S.},
  title={Camassa-Holm, Korteweg-de Vries and related models for water waves},
  journal={J. Fluid Mech.},
  volume={455},
  date={2002},
  pages={63--82},
    }


\bib{McKean04}{article}{
   author={McKean, H.P.},
   title={Breakdown of the Camassa-Holm equation},
   journal={Comm. Pure Appl. Math.},
   volume={57},
   date={2004},
   number={3},
   pages={416--418},  
}

\bib{Yin03}{article}{
   author={Yin, Z.},
   title={Global existence for a new periodic integrable equation},
   journal={J. Math. Anal. Appl.},
   volume={283},
   date={2003},
   pages={129--139},
}



\end{thebibliography}
\end{document}